\documentclass[onefignum,onetabnum]{siamart171218}
\usepackage{subcaption}
\usepackage{graphicx}
\usepackage{diagbox}
\usepackage{cite}
\usepackage{hyperref}
\usepackage{tablefootnote}

\usepackage{amsmath,amssymb,amsfonts}
\usepackage{algpseudocode}
\usepackage{algorithm}
\usepackage{algorithmicx}
\usepackage{multirow}

\newcommand{\eps}{\varepsilon}
\newcommand{\tv}{\tilde{V}}
\newcommand{\ts}{\tilde{S}}

\newcommand{\tl}{\tilde{L}}
\newcommand{\tu}{\tilde{U}}
\newcommand{\kb}{\tilde{\kappa}_F(B)}
\newcommand{\kbp}{\tilde{\kappa}_F^p(B)}

\newcommand{\bv}{\bar{V}}
\newcommand{\bl}{\bar{L}}

\newcommand{\vect}{\mathbf}
\newcommand{\ignore}[1]{}

\title{Neumann Series in GMRES and Algebraic Multigrid Smoothers}

\author{
Stephen~Thomas%
\thanks{National Renewable Energy Laboratory, Golden, Colorado}
\and Arielle Carr\thanks{Lehigh University, Bethlehem, PA}
\and Paul Mullowney%
\footnotemark[1]
\and Ruipeng Li%
\thanks{Lawrence Livermore National Laboratory, Livermore CA}
\and Kasia~\'{S}wirydowicz%
\thanks{Pacific Northwest National Laboratory, Richland, WA}
}

\begin{document}

\maketitle

\begin{abstract}
Neumann series underlie both Krylov methods and algebraic multigrid 
smoothers. A low-synchronization modified Gram-Schmidt (MGS)-GMRES
algorithm is described that employs a Neumann series 
to accelerate the projection step.
A corollary to the backward stability result of Paige et al.~\cite{2006--simax--paige-rozloznik-strakos} demonstrates that the 
truncated Neumann series approximation is sufficient for convergence of 
GMRES under standard conditions. The lower triangular solver associated
with the correction matrix $T_m = (\: I + L_m \:)^{-1}$
may then be replaced by a matrix-vector product with $T_m = I - L_m$.
Next, Neumann series are applied to accelerate
the classical R\"uge-Stuben algebraic multigrid (AMG) preconditioner 
using either a polynomial Gauss-Seidel or incomplete ILU smoother.
Here, the sparse triangular solver employed 
in the Gauss-Seidel and ILU smoothers is replaced
by an inner iteration based upon matrix-vector products.
Henrici's departure from normality of the associated iteration matrices
leads to a better understanding of these series.  
Additionally, connections are made between the (non)normality of 
the $L$ and $U$ factors and nonlinear stability analysis, 
as well as the pseudospectra of the coefficient matrix.  
Furthermore, re-orderings that preserve structural symmetry also reduce 
the departure from normality of the upper triangular factor and improve 
the relative residual of the triangular solves.
To demonstrate the effectiveness of this approach on
many-core architectures, the proposed solver and preconditioner 
are applied to the pressure continuity equation for the
incompressible Navier-Stokes equations of fluid motion.
The pressure solve time is reduced considerably without a change 
in the convergence rate and the polynomial Gauss-Seidel smoother
is compared with a Jacobi smoother. Numerical and timing results are
presented for Nalu-Wind and the PeleLM combustion codes, where ILU
with iterative triangular solvers is shown 
to be much more effective than polynomial Gauss-Seidel.
\end{abstract}
   
\section{Introduction}
The purpose of the present work is to show the important role
Neumann series play in MGS-GMRES and C-AMG smoothers. By
revealing this structure and analyzing it, we devise ways
to speed-up the GMRES+AMG solver in hypre \cite{Falgout2004}.  
The generalized minimal residual (GMRES) Krylov subspace method \cite{Saad86}
is often employed to solve the large linear systems arising in 
high-resolution physics based simulations using the incompressible
Navier-Stokes equations in fluid mechanics. \'{S}wirydowicz et al.
\cite{2020-swirydowicz-nlawa} recently improved the parallel 
strong-scaling of the algorithm by reducing the 
MPI communication requirements to a minimum, 
while maintaining the numerical stability and robustness
of the original algorithm. In order to achieve fast convergence
of an elliptic solver, such as the pressure continuity equation,
the classical Ruge-St\"uben \cite{Ruge1987} algebraic multigrid 
(C-AMG, or simply AMG) solver is applied as a preconditioner. 
The low-synchronization MGS-GMRES iteration, Gauss-Seidel and ILU
smoothers employ direct triangular solvers.\footnote{Note that the phrase {\it direct triangular solver} 
is used to differentiate between 
the iterative approximations versus the traditional 
forward and backward recurrences.} 
The triangular solve in GMRES is relatively small 
and local to each MPI rank. For Gauss-Seidel the coarse level
matrices in the AMG $V$-cycle
are distributed and some communication is required. 
Triangular solvers are, in general, difficult 
to implement in parallel on many-core architectures. In this study, 
the associated iteration matrices are replaced by truncated Neumann 
series expansions
leading to polynomial-type smoothers.
They result in a highly efficient and backward stable approach for 
Exascale class supercomputers.

Consider the large, sparse linear systems arising from the 
discretization of the incompressible Navier-Stokes pressure 
continuity equation \cite{Ashesh2021}.
In the present study, linear systems of the form $A\vect{x}=\vect{b}$ with $A$ an $n \times n$ real-valued matrix, are solved with Krylov subspace methods
using one $V$-cycle of the C-AMG  algorithm as the 
preconditioner.  Here, let 
$\vect{r}^{(0)} = \vect{b} - A\vect{x}^{(0)}$ 
denote the initial residual with initial guess $\vect{x}^{(0)}$.
Inside GMRES, the Arnoldi $QR$ algorithm is applied to generate an orthonormal
basis for the Krylov subspace 
${\cal K}_m(B)$ spanned by the columns of the $n\times m$ matrix, $V_m$, 
where $m \ll n$,
and produces the $(m+1)\times m$ Hessenberg matrix, $H_{m+1,m}$, 
in the Arnoldi expansion such that%
\[
AV_m = V_{m+1}H_{m+1,m}.
\]
In particular, the Arnoldi algorithm produces a $QR$ factorization of 
$B = [\:\vect{r}^{(0)},\:AV_m\:]$ and the columns of $V_{m+1}$ 
form an orthogonal basis for the Krylov subspace ${\cal K}_{m+1}(B)$ \cite{2006--simax--paige-rozloznik-strakos}. 
The orthogonality of the columns
determines the convergence of Krylov methods for linear system solvers. 
However, in finite-precision arithmetic, $V_m$ may 
``lose'' orthogonality and this loss, 
as measured by $\|I - V_m^TV_m\|_F$, 
may deviate substantially from machine precision, ${\cal
O}(\eps)$, (see Paige and Strakos \cite{Paige2002}).  
When linear independence is completely lost, the Krylov iterations
may fail to converge. For example, the MGS-GMRES iteration will stall 
in this case, and this occurs when $\|S_m\|_2 = 1$,
where the matrix $S_m = (I+U_m)^{-1}U_m$ was introduced in the seminal backward stability 
analysis of Paige et al.~\cite{2006--simax--paige-rozloznik-strakos}. 
Here, $U_m$ is the $m \times m$ strictly upper triangular part of $V_m^TV_m$.

The development of low-synchronization Gram-Schmidt and generalized 
minimal residual algorithms by  
\'{S}wirydowicz et al.~\cite{2020-swirydowicz-nlawa}  and 
Bielich et al.~\cite{DCGS2Arnoldi} was largely 
driven by applications that need stable, yet scalable solvers. 
Both the modified (MGS) and classical Gram-Schmidt with 
re-orthogonalization 
(CGS2) are stable algorithms for a GMRES solver. Indeed,
CGS2 results in an ${\cal O}(\eps)$ loss of orthogonality, 
which suffices for GMRES to converge.
Paige et al.~\cite{2006--simax--paige-rozloznik-strakos} show that despite
${\cal O}(\eps)\kappa(B)$ loss of orthogonality, MGS-GMRES is backward
stable for the solution of linear systems. Here, the condition number
of the matrix $B$ is given by
$\kappa(B) = \sigma_{\max}(B)/\sigma_{\min}(B)$, where $\sigma_{\max}(B)$ and $\sigma_{\min}(B)$ are
the maximum and minimum singular values of the matrix $B$, respectively.

An inverse compact $WY$ (ICWY) modified Gram-Schmidt algorithm is 
presented in~\cite{2020-swirydowicz-nlawa} and is 
based upon the application of a projector
\[
P = I - V_m \: T_m \: V_m^T, \quad T = (\: V_m^TV_m \:)^{-1}
\]
where $V_m$ is again $n\times m$, $I$ is the identity matrix of dimension $n$,
and $T_m$ is an $m\times m$ correction matrix.
To obtain a one-reduce MGS algorithm, or one MPI global 
reduction per GMRES iteration, the normalization
is delayed to the next iteration. The correction matrix
$T_m$ is obtained from  
the strictly lower triangular part of $V_m^TV_m$, denoted $L_m$.  Note that 
because $V_m$ has almost orthonormal columns, the norm of $L_m$ is small, 
and $T_m$ is close to $I$ (here, the identity matrix of dimension $m$).

In this work, an alternative to computing the inverse of $T_m$ 
via triangular solve is developed. 
A Neumann series expansion for the inverse of the lower triangular 
correction matrix, $T_m$, results in the
ICWY representation of the projector $P$ for
the modified Gram-Schmidt factorization of the Arnoldi matrix $B$. This is written as
\begin{equation}\label{eq:GS}
P = I - V_m \: T \: V_m^T, \quad
T_m = (\: I + L_{m} \:)^{-1}  = I - L_{m} + L_{m}^2 - \cdots + L_{m}^p,
\end{equation}
where the columns of $L_m$ are defined by the matrix-vector products $V_{m-2}^T\vect{v}_{m-1}$.
The sum is finite because the matrix $L_m$ is
nilpotent, as originally noted by Ruhe \cite{Ruhe83}. 

In this paper, a corollary
to the backward stability result of Paige et al.~\cite{2006--simax--paige-rozloznik-strakos} 
and Paige and Strako\v{s} \cite{Paige2002} is presented that
demonstrates  matrix-vector multiplication by
$T_m=I-L_m$ in the projection step is sufficient for convergence of MGS-GMRES when 
$\|L_m\|_F^p = {\cal O}(\eps^p)\kappa_F^p(B)$, where $p > 1$. 
A new formulation of GMRES based upon the
truncated Neumann series for the matrix $T_m$ is
presented. In particular, the loss of orthogonality and
backward stability of the truncated and MGS-GMRES algorithms
are the same because of the above corollary.
For extremely ill-conditioned matrices, the convergence history of
the new algorithm has been found to be identical to
the original GMRES algorithm introduced by Saad and Schultz ~\cite{Saad86}.
In particular, convergence is achieved when the 
norm-wise relative backward error reaches machine precision (NRBE); see Section \ref{sec:lowsynch} for more details. For this reason, 
Paige and Strako\v{s} \cite{Paige2002}
recommended that the NBRE be applied as the stopping criterion.

A polynomial Gauss-Seidel iteration 
can also be interpreted with respect to the
degree-$p$ Neumann series expansion as described in the recent
paper by Mullowney  et al.~\cite{Mullowney2021}.
Given the regular matrix splitting 
$A = M - N$, $A = D + L + U$, $M = D + L$ and $N = - U$,\footnote{Note that 
by convention, this matrix splitting is written such that $L$ is the strictly 
lower triangular part of $A$, and $U$ is the strictly upper triangular part.  The 
matrices $L$ and $U$ in 
Gauss-Seidel smoother 
are different from the those used in Gram-Schmidt projection in (\ref{eq:GS}).} the
polynomial Gauss-Seidel iteration is given by
\[
\vect{x}^{(k+1)} := \vect{x}^{(k)} + (I+D^{-1} L)^{-1}\:D^{-1} \vect{r}^{(k)} = 
\vect{x}^{(k)} + \sum_{j=0}^{p}
(-D^{-1} L)^j D^{-1} \vect{r}^{(k)} 
\]
For the ILU factorization \cite{ChowSaad1997}, 
the $L = I + L_s$ factor has a unit diagonal and strictly 
lower triangular part $L_s$. Thus, the inverse $L^{-1}$ may be expressed as
a Neumann series and, with $L_s$ nilpotent, this is a finite sum. Similarly,
for the strictly upper triangular factor $U_s$ such that $U = I + U_s$, after either row  or row/column scaling, 
the inverse can be written as a Neumann series, as follows
\[
(\: I + U_s \:)^{-1} = I - U_s + U_s^2 - \cdots
\]
Furthermore, when $\|U_s^p\|_2 < 1$, for $p\ge 1$ the 
series converges. In the present study, the above iterations are 
applied as smoothers when constructing a C-AMG preconditioner.

In \cite{Thomas2021}, a method for scaling the $U$ factor to form $U = I + U_s$
in order to mitigate a high degree of
non-normality and condition number is proposed.  
In general, fast convergence of the Jacobi 
iteration can then be expected \cite{Anzt2015} when 
iteratively computing the triangular solution, thus
avoiding the comparatively high cost on the GPU of 
a direct triangular solve.  Here, the analysis in
\cite{Thomas2021} is extended and theoretical upper bounds on the departure from normality
\cite{Henrici1962} of a (row and/or column) scaled triangular matrix are provided.  In the case of
nonsymmetric coefficient matrices, reordering prior to scaling may be employed to enforce a 
symmetric sparsity pattern.  Following the analysis in \cite{ChowSaad1997}, an empirical 
relationship between the drop tolerance and fill level used when computing an incomplete 
$LU$ factorization and the subsequent departure from normality of the  triangular factors 
following row and/or column scaling is established.  

An additional contribution of our paper is to combine three different 
Neumann series to create fast and robust solvers. The time to solution 
is the most important metric versus the number of solver iterations taken
in fluid mechanics simulations. 
The low-synchronization MGS-GMRES may take a larger
number of iterations at little or no extra cost if the preconditioner
is less expensive. However, the triangular solvers employed
by smoothers are not efficient on multi-core architectures. Indeed,
the sparse matrix-vector product (SpMV) is 25--50$\times$ faster
on GPU architectures \cite{Anzt2015}.
This observation prompted the introduction of polynomial 
AMG smoothers in~\cite{Mullowney2021}, 
which significantly lowers the cost of the $V$-cycle. 

The paper is organized as follows. Low synchronization and generalized
minimal residual algorithms are discussed in Section 2. A corollary to
Paige et al. \cite{2006--simax--paige-rozloznik-strakos} in Section 3
establishes that $T_m = I-L_m$ is sufficient for convergence.
In Section 4, a variant of MGS-GMRES is presented that uses this $T_m$.
Section 5 describes recent developments for hypre on GPUs.
The polynomial Gauss-Seidel and ILU smoothers are then introduced in 
Section 5 where it is shown that pivoting leads to 
rapid convergence of iterative triangular solvers.
The departure from normality of the factors can also be
mitigated by either row or row and column scaling
via the Ruiz algorithm, thus avoiding divergence of the iterations.

{\it Notation} Lowercase bold letters denote a column vector 
and uppercase letters are matrices (e.g., $\vect{v}$ and $A$,
respectively).  $a_{ij}$ represents the $(i,j)$ scalar entry 
of a matrix $A$, and $\vect{a}_j$ denotes the $j^{th}$ column of $A$.
Where appropriate, for a matrix $A$, $A_j$ is the $j$--th column.
Superscripts indicate the approximate solution 
(e.g., $\vect{x}^{(k)}$) and corresponding residual (e.g.,
$\vect{r}^{(k)}$) of an iterative method at step $k$. Throughout this
paper, we will explicitly refer to {\it strictly upper/lower triangular
matrices} and use the notation $U_k$ (or $L_k$) and $U_s$ 
(or $L_s$).\footnote{We note that the distinction between these 
two notations in crucial.  For $U_k$, the size of the strictly upper
triangular matrix changes with $k$, whereas the size of $U_s$ remains
fixed.}  Vector notation indicates a subset of the rows
and/or columns of a matrix; e.g., $V_{1:k+1,1:k}$ denotes the first 
$k+ 1$ rows and $k$ columns of the matrix $V$ and the
notation $V_{:,1:k}$ represents the entire row of the first $k$ columns 
of $V$.  $H_{m+1,m}$ represents an $(m+1)\times m$ matrix, and in particular 
$H$ refers to a Hessenberg matrix.  In cases where standard notation in 
the literature is respected that may otherwise conflict with the aforementioned notation, 
this will be explicitly indicated.

\section{Low-Synchronization Gram-Schmidt Algorithms}

Krylov linear system solvers are often required for extreme
scale physics simulations on parallel
machines with many-core accelerators.  
Their strong-scaling is limited by the number and frequency of
global reductions in the form of {\tt MPI\_AllReduce} and these communication
patterns are expensive.  
Low-synchronization algorithms are
designed such that they require only one reduction per iteration 
to normalize each vector and apply projections. 

A review of compact $WY$ Gram Schmidt algorithms and their computational
costs is given in~\cite{DCGS2Arnoldi}.
The ICWY form for MGS is the lower triangular matrix $T_k=(I+L_k)^{-1}$  
and was derived in \'{S}wirydowicz et al.~\cite{2020-swirydowicz-nlawa}. 
Here, $I$ denotes the identity matrix and $L_k$ is strictly lower triangular.
Specifically, these algorithms batch the inner-products together and 
compute one row of $L_k$ as
\begin{equation}\label{eq:matvec}
L_{k-1,1:k-2} = (\:V_{k-2}^T\:\vect{v}_{k-1}\:)^T.  
\end{equation}
The resulting ICWY projector $P$ is given by
\[
P =  I - V_{k-1}\:T_{k-1}\:V_{k-1}^T, \quad 
T_{k-1} = (\:I + L_{k-1}\:)^{-1}
\]
The implied triangular solve requires an additional $(k-1)^2$ flops at
iteration $k-1$ and thus leads to a slightly higher operation count 
compared to the original MGS algorithm.  The matrix-vector
multiply in (\ref{eq:matvec}) increases the complexity by $mn^2$  ($3mn^2$ total) 
but decreases the number of global reductions from $k-1$ at iteration $k$ to 
only one reduction when combined with the lagged normalization of a Krylov
vector. 


\section{Loss of Orthogonality}

When the Krylov vectors are orthogonalized via the finite 
precision MGS algorithm, and their loss of orthogonality is related in a 
straightforward way to the convergence of GMRES. 
In particular, orthogonality among the Krylov 
vectors is effectively maintained until the norm-wise relative 
backward error approaches the machine precision as discussed
in Paige and Strako\v{s} \cite{Paige2002}.
In this section, a corollary to 
Paige et al.~\cite{2006--simax--paige-rozloznik-strakos} 
establishes that 
$T_k = (I + L_{k})^{-1} = I - L_{k} + {\cal O}(\eps^2)\kappa^2(B)$,
where $B = [\: \vect{r}^{(0)}, AV_k\:]$ for the ICWY MGS formulation 
of the Arnoldi expansion.

Let $A$ be an $n\times n$ real-valued matrix, and consider the
Arnoldi factorization of the matrix $B$.
After $k$ steps, in exact arithmetic, the algorithm
produces the factorization 
\[
AV_{k} = V_{k+1}\:H_{k+1,k}, \quad V_{k+1}^TV_{k+1} = I_{k+1}
\]
where $H_{k+1,k}$ is an upper Hessenberg matrix. 
When applied to the linear system $A\vect{x} = \vect{b}$, 
assume $\vect{x}^{(0)} = 0$, $\vect{r}^{(0)} = \vect{b}$, 
$\|\vect{b}\|_2 = \rho$ and $\vect{v}_1 = \vect{b} / \rho$.
The Arnoldi algorithm produces an orthogonal basis for the Krylov vectors 
spanned by the columns of the matrix $V_{k}$.
Following the notation in \cite{2006--simax--paige-rozloznik-strakos}, 
the notation $\overline{A}_{k}$ is employed to represent the computed matrix in 
finite-precision arithmetic with $k$ columns, and $\tilde{A}_{k}$ to represent 
the properly normalized matrix with $k$ columns.  We refer the reader to
\cite{2006--simax--paige-rozloznik-strakos} for more details on these definitions.  
Consider the computed matrix $\bv_k$ with Krylov vectors as columns.  The
strictly lower triangular matrix $\bl_{k}$ is obtained from the loss of
orthogonality relation
%
\[
\bv_{k}^T\bv_{k} = I + \bl_{k} + \bl_{k}^T.
\]
%
Let us first consider the lower triangular solution algorithm for $T_k$
where the elements of $L_{k}$ appear in the correction matrix $T_k$, along 
with higher powers of the inner products 
in the Neumann series. 


A corollary to the backward stability results
from Paige and Strako\v{s} \cite{Paige2002}, and 
Paige et al.~\cite{2006--simax--paige-rozloznik-strakos} 
is now established,
namely that the Neumann series for $T_k$ may be truncated
according to
\[
T_k = (\: I + \tl_{k} \:)^{-1} = I - \tl_{k} + {\cal O}(\eps^2)\kappa^2(B).
\]
The essential result is based on the $QR$ factorization of the matrix
\[
B = [\: \vect{r}^{(0)}, \: AV_k \: ] = V_{k+1}
\left[ \: \vect{e}_1 \rho, \: H_{k+1,k} \: \right].
\]
A bound on the loss-of-orthogonality given by
\[
\|\:I - \tv_{k+1}^T\tv_{k+1}\:\|_F \le \kappa (\:
[ \: \vect{r}^{(0)}, \: AV_k \:]\:)\:{\cal O}(\eps),
\]
and the derivation of this result allows us to find an upper
bound for $\|\tl_k\|_F^p$.

The upper triangular structure of the loss of orthogonality
relation is revealed to be
\[
\tu_k = (\: I - \ts_k \: )^{-1}\: \ts_k = \ts_k\: (\: I - \ts_k \:)^{-1}
\]
and  it follows that $\tl_k = \tu_k^T$,
\begin{eqnarray*}
I - \ts_k^T & = & I - \tl_k\:( \: I + \tl_k\: )^{-1} \\
            & = & I - \tl_k \: (\: I - \tl_k + \tl_k^2 - \tl_k^3 + \cdots \:) \\
            & = & ( I + \tl_k )^{-1}
\end{eqnarray*}
where $\tu_k = \tl_k^T$ is the strictly upper triangular part of $\tv^T\tv$.
It then follows immediately, that a bound on the loss of orthogonality is given by
\[
\sqrt{2} \|\tv_k^T\tilde{\vect{v}}_k\|_2 \le \| I - \tv_k^T\tv_k \|_2 = \sqrt{2} \|(\:I - \ts_k\:)^{-1}\ts_k\:\|_F
\le \frac{4}{3}(2m)^{1/2}\hat{\gamma}_n\tilde{\kappa}_F(B)
\]
where
\[
\hat{\gamma}_n = \frac{\tilde{c}_n\eps}{1 - \tilde{c}_n\eps}, \quad
\tilde{\kappa}_F(B) = \min_{{\rm diag} D>0} \|AD\|_F / \sigma_{\min} (AD)
\]
The matrix $D$ is defined in \cite{2006--simax--paige-rozloznik-strakos}
to be {\em any} positive definite diagonal matrix. Therefore, it follows that
\[
\|\tu_k\|_F = \|\tl_k^T\|_F \le {\cal O}(\eps)\kb, \quad \|\tl_k\|_F^p \le {\cal O}(\eps^p) \kbp
\]
and thus the matrix inverse from the Neumann series is given by
\begin{eqnarray*}
T_k = (\: I + \tl_k \: )^{-1} & = & I - \tl_k + \tl_k^2 - \tl_k^3 + \cdots \\
                        & = & I - \tl_k + {\cal O}(\eps^2)\kappa^2(B)
\end{eqnarray*}

The growth of the condition number above is related to the norm-wise relative backward error
\[
\beta(\vect{x}^{(k)}) = \frac{\|\vect{r}^{(k)}\|_2}{\|\vect{b}\|_2 + \|A\|_\infty\|\vect{x}^{(k)}\|_2}
\]
and in particular, 
%
$\beta(\vect{x}^{(k)}) \: \kappa( \left[ \: \vect{r}^{(0)}, \: AV_k \: \right] ) = {\cal O}(1)$
%

\section{Low-synchronization MGS-GMRES}\label{sec:lowsynch}

The MGS--GMRES orthogonalization algorithm is the $QR$ factorization of
a matrix $B$ formed by adding a new column to $V_k$ in each iteration
\[
\left[ \begin{array}{cc} \vect{r}^{(0)}, & AV_m\end{array} \right] =
V_{m+1} \left[ \begin{array}{cc} \|\vect{r}^{(0)}\|\vect{e}_1, & H_{m+1,m} \end{array} \right].
\]
We remind the reader that bold-face with
superscripts  (e.g., $\vect{r}^{(0)}$) denote the residual in the GMRES algorithm 
and subscripting (e.g., $r_{1:i+1,i+2}$) denotes the corresponding element in the upper
triangular matrix $R$.  

The MGS-GMRES algorithm was proven to be backward stable for the solution of 
linear systems $A\vect{x} = \vect{b}$ 
in  \cite{2006--simax--paige-rozloznik-strakos} and 
orthogonality is maintained to ${\cal O}(\eps)\kappa(B)$, 
depending upon the condition number of the matrix $B = \kappa([\vect{r}^{(0)},\:AV_k])$.  
The normalization of the Krylov vector $\vect{v}_{i+1}$ 
at iteration $i+1$ represents
the delayed scaling of the vector $\vect{v}_{i+2}$ in the
matrix-vector product $\vect{v}_{i+2} = A\vect{v}_{i+1}$. Therefore,
an additional Step 8 is required in the one-reduce 
Algorithm \ref{fig:onesynch},
$r_{1:i+1,i+2} = r_{1:i+1,i+2} / r_{i+1,i+1}$ 
and $\vect{v}_{i+1} = \vect{v}_{i+1} / r_{i+1,i+1}$.
The diagonal element of the $R$ matrix, corresponds to $H_{i+1,i}$,
in the Arnoldi $QR$ factorization of the matrix
$B$, and is updated after the MGS projection in
Step 12 of the GMRES Algorithm \ref{fig:onesynch}.
The higher computational speed on a GPU is achieved when the 
correction matrix is replaced by $T_{i+1} = (\:I - L_{i+i}\:)$ in Step 11 of the 
GMRES Algorithm \ref{fig:onesynch}, resulting in  a matrix-vector multiply
\begin{algorithm}[htb!]
\centering
\begin{algorithmic}[1]
\Statex{{\bf Input:} Matrix $A$; right-hand side vector $\vect{b}$; initial guess vector $\vect{x}^{(0)}$} 
\Statex{{\bf Output:} Solution vector $\vect{x}$} 
  \State{$\vect{r}^{(0)} = \vect{b} - A\vect{x}^{(0)}$, $\vect{v}_1 = \vect{r}^{(0)}$.}
  \State{$\vect{v}_{2} = A\vect{v}_{1}$ } 
  \State{$(\: V_{2},\: R,\: L_2\: ) = {\rm mgs} (\: V_{2},\: R,\: L_1\: )$ }
  \For {$i=1, 2, \ldots, m$}
     \State{$\vect{v}_{i+2} = A\vect{v}_{i+1}$ \Comment{Matrix-vector product}} 
     \State{$ [\: L_{:,i+1}^T,\: \vect{r}_{i+2}\:] =  V_{i+1}^T [\vect{v}_{i+1}\, \vect{v}_{i+2}]$ \Comment{Global synchronization} }
     \State{$r_{i+1,i+1} = \|\vect{v}_{i+1}\|_2$ }
     \State{$\vect{v}_{i+1} = \vect{v}_{i+1}/r_{i+1,i+1}$ \Comment{Lagged normalization}}
     \State{$r_{1:i+1,i+2} = r_{1:i+1,i+2}/r_{i+1,i+1}$ \Comment{Scale for Arnoldi}}
     \State{$L_{:,i+1}^T = L_{:,i+1}^T/r_{i+1,i+1}$  }
     \State{$r_{1:i+1,i+2} = T_{i+1}\:r_{1:i+1,i+2}$ \Comment{ Projection Step}}
     \State{$\vect{v}_{i+2} = \vect{v}_{i+2} - V_{i+1} \: r_{1:i+1,i+2}$ }
     \State{$H_{i} = \vect{r}_{i+1}$}
     \State{Apply Givens rotations to $H_{i}$ }
  \EndFor
  \State{$\vect{y}_m = {\rm argmin} \|(\: H_{m}\vect{y}_m - \|\vect{r}_0\|_2\vect{e}_1 \:) \|_2$ }
  \State{$\vect{x} = \vect{x}^{(0)} + V_m\vect{y}_m$ }            
\end{algorithmic}
\caption{\label{fig:onesynch} Truncated Neumann Series MGS-GMRES}
\end{algorithm}

The most common convergence criterion when solving linear systems 
of the form $A\vect{x} = \vect{b}$ in existing iterative
solver frameworks is based upon the relative residual,
\begin{equation}
\frac{\|\vect{r}^{(k)}\|_2}{\|\vect{b}\|_2} =
\frac{\|\vect{b} - A\vect{x}^{(k)}\|_2}{\|\vect{b}\|_2} < tol,
\label{eq:relres}
\end{equation}
where $tol$ is some user-defined convergence tolerance. 
However, when the columns of $V_k$ become linearly dependent,
as indicated by $\|S_k\|_2=1$, the orthogonality of the Krylov
vectors is completely lost. Then the convergence of MGS-GMRES
flattens or stalls at this iteration. Due to the relationship 
with the backward error for solving
linear systems $A\vect{x} = \vect{b} $, elucidated 
by Prager and Oettli \cite{Prager64}
and Rigal and Gaches \cite{Rigal67}, the backward stability
analysis of Paige et al.~\cite{2006--simax--paige-rozloznik-strakos}
relies instead upon the norm-wise relative backward error (NRBE)
reaching machine precision.

\section{Algebraic Multigrid Preconditioner}

Neumann series also naturally arise when employing certain smoothers in AMG, and in particular, polynomial Gauss-Seidel and ILU smoothers.  Before analyzing this, necessary background information on AMG is first provided.

AMG \cite{Ruge1987} is an  effective method for solving  
linear systems of equations. 
In theory, AMG can solve a linear system with $n$
unknowns in ${\cal O}(n)$ operations, and even though it was originally developed as a solver, it is now common practice to use AMG as
a preconditioner to a Krylov method; in particular, when such Krylov solvers are applied
to large-scale linear systems arising in physics-based simulations.
As either a solver or preconditioner, an AMG method accelerates the solution of a linear system $A\vect{x} = \vect{b}$ 
through error reduction by using a sequence of coarser matrices 
called a {\em hierarchy}.  
These matrices will be denoted $A_{k}$, where
$k=0\dots m$, and $A_0 = A$.  Each $A_{k}$ has
dimensions $m_k \times m_k$ where $m_k > m_{k+1}$ for $k<m$.  
For the purposes of this paper, assume that
\begin{equation}
A_k=R_k A_{k-1} P_k\,,
\end{equation}
for $k>0$, where $P_k$ is a rectangular matrix with dimensions $m_{k-1} \times
m_k$.  $P_k$ is referred to as a {\em prolongation matrix} or {\em
prolongator}. $R_k$ is the {\em restriction matrix} and $R_k = P_k^T$ in the
Galerkin formulation of AMG.  Associated with each $A_k$, $k<m$, is a solver
called a {\em smoother}, which is usually a stationary iterative method,
e.g. Gauss-Seidel, polynomial, or incomplete factorization.  
The {\em coarse solver} used for $A_m$ (the lowest level in the hierarchy) 
is often a direct solver,  although it may be an
iterative method if $A_m$ is singular.

The setup phase of AMG is nontrivial for several reasons.
Each prolongator $P_k$ is developed algebraically from $A_{k-1}$.
Once the transfer matrices are determined, the
coarse-matrix representations are computed recursively from $A$ through
sparse triple-matrix multiplication.
This describes a $V$-cycle, the simplest complete AMG cycle; refer to
Algorithm 1 in the recent paper by Thomas et al. \cite{Thomas2021}
for a complete description.  AMG methods achieve
optimality (constant work per degree of freedom in $A_0$) through 
error reductions by the smoother and corrections propagated from
coarser levels.


\subsection{Ruge--St\"uben Classical AMG}

\label{sec:camg}
We now give an overview of classical Ruge-St\"{u}ben AMG, starting
with notation. Interpolation is formulated in terms of matrix operations.
Point $j$ is a neighbor of $i$ if and only if there is a non-zero element
$a_{ij}$ of the matrix $A$.
Point $j$ strongly influences $i$ if and only if
\begin{equation}
|a_{ij}| \ge \theta \max_{k\ne i}\:|\:a_{ik}\:| \,,
\end{equation}
where $\theta$ is the strength of connection threshold, $0 < \theta \le 1$. 
This strong influence relation is used to select coarse points. 
These points are retained in the next coarser level, and the remaining fine points
are dropped. Let $C_k$ and $F_k$ be the coarse and fine points selected at
level $k$, and let $m_k$ be the number of grid points at level $k$ ($m_0 = n$).
Then, $m_k = |C_k| + |F_k|$, $m_{k+1} = |C_k|$, $A_k$ is a $m_k\times m_k$
matrix, and $P_k$ is an $m_{k-1}\times m_k$ matrix. Here, the coarsening is
performed row-wise by interpolating between coarse and fine points.
The coarsening generally attempts to fulfill two contradictory criteria.
In order to ensure that a chosen interpolation scheme is well-defined and of
good quality, some close neighborhood of each fine point must contain a
sufficient amount of coarse points to interpolate from. Hence the set of coarse
points must be rich enough. However, the set of coarse points should be
sufficiently small in order to achieve a reasonable coarsening rate.
The interpolation should lead to a reduction of roughly five times
the number of non-zeros $nnz$ at each level of the $V$-cycle.



During the setup phase of AMG methods, 
the multi-level $V$-cycle hierarchy is constructed 
consisting of linear systems having 
exponentially decreasing sizes on coarser levels.
A strength-of-connection matrix $S$, 
is typically first computed to
indicate directions of algebraic smoothness 
and applied in the coarsening algorithms.
The construction of $S$ may be performed efficiently, 
because each row is computed independently by
selecting entries in the corresponding row of 
$A$ with a prescribed threshold value $\theta$.
hypre-BoomerAMG provides the parallel maximal 
independent set (PMIS) coarsening \cite{DeSterck2006}, 
which is modified from Luby's algorithm \cite{Luby1986} 
for finding maximal independent sets using random numbers. 

Interpolation operators in AMG transfer residual errors between adjacent levels.
There are a variety of interpolation schemes available.
Direct interpolation \cite{StubenAMG} is straightforward to 
implement in parallel because the interpolatory
set of a fine point $i$ is just a subset of the neighbors of $i$, and thus
the interpolation weights can be determined solely by the $i$-th equation.
A bootstrap AMG (BAMG) \cite{Manteuffel2010} variant of direct interpolation is
generally found to be better than the original formula. 
The weights $w_{ij}$ are computed by
solving the local optimization problem
\[ 
\min \|a_{ii}\vect{w}_i^T + a_{i,C_i^s}\|_2 \quad \mathrm{s.t.} \quad
\vect{w}_i^T \vect{f}_{C_i^s} = \vect{f}_i,
\label{eq:bamgdir}
\]
where $\vect{w}_i$ is a vector that contains $w_{ij}$, $C_i^s$
and denotes strong C-neighbors of $i$ and $\vect{f}$ is a target 
vector that needs to be interpolated exactly.
For elliptic problems where the near null-space is spanned 
by constant vectors, i.e.,  $\vect{f}=\vect{1}$, the
closed-form solution of \eqref{eq:bamgdir} is given by
\begin{equation}
w_{ij} = - \frac{a_{ij} + \beta_i/n_{C_i^s}} {a_{ii}+\sum_{k \in N_i^w} a_{ik}}, 
\quad \beta_i=\sum_{k\in \{\vect{f}_i \cup C_i^w\}}a_{ik} \ ,
\end{equation}
where $n_{C_i^s}$ denotes the number of points in $C_i^s$, $C_i^w$ the weak C-neighbors of $i$, $\vect{f}_i$ the F-neighbors, and $N_i^w$ the weak neighbors.
A known issue of PMIS coarsening is that it can result in F-points without C-neighbors \cite{DeSterck2008}. 
In such situations, distance-one interpolation algorithms often work well, 
whereas interpolation operators 
that can reach C-points at a larger range, such as 
the extended interpolation \cite{DeSterck2008}, 
can generally yield much better convergence.  However, 
implementing extended interpolation is more complicated because
the sparsity pattern of the interpolation operator cannot be
determined a priori, which requires dynamically combining C-points 
in a distance-2 neighborhood.
With minor modifications to the original form, it turns out that
the extended interpolation operator can be rewritten by using
standard sparse matrix computations such as matrix-matrix (M-M) 
multiplication and diagonal scaling with 
certain $FF$- and $FC$-sub-matrices. 
The coarse-fine C-F splitting of the coarse matrix $A$
is given by
\[
A =
\left[
\begin{array}{cc}
A_{FF} &  A_{FC} \\
A_{CF} &  A_{CC}
\end{array}
\right]
\]
where $A$ is assumed to be decomposed into $A=D+A^s+A^w$, the diagonal, 
the strong part and weak part respectively, and
$A_{FF}^w$, $A_{FC}^w$, $A_{FF}^s$ and $A_{FC}^s$ are the corresponding 
sub-matrices of $A^w$ and $A^s$.

The full prolongation operator is given by
\[
P =
\left[
\begin{array}{c}
W  \\
I
\end{array}
\right]
\]
The extended ``MM-ext'' interpolation now takes the form
\begin{align*}
   W &= -\big[{(D_{FF}+D_{\gamma})^{-1} (A^s_{FF}  + D_{\beta})}\big] \big[{D_{\beta}^{-1} A^s_{FC}}\big]
\end{align*}
with
\[
D_{\beta} = {\rm diag}(A^s_{FC} \vect{1}_C) \quad
D_{\gamma} = {\rm diag}(A^\vect{w}_{FF}\vect{1}_F+A^\vect{w}_{FC}\vect{1}_C),
\]
This formulation allows simple and efficient implementations.
Similar approaches that are referred to as ``MM-ext+i'' modified from
the original extended+i algorithm  \cite{DeSterck2008} and
``MM-ext+e'' are also available in BoomerAMG.
See \cite{Li2020} for details on the class of M-M based interpolation operators.

Aggressive coarsening reduces the grid and operator complexities of the AMG hierarchy,
where a second coarsening is applied to the C-points obtained from
the first coarsening to produce a smaller set of final C-points. 
The A-1 aggressive coarsening strategy described in \cite{StubenAMG}
is employed in our studies.
The second PMIS coarsening is performed with the $CC$ block of  $S^{(A)}=S^2+S$ that    
has nonzero entry $S^{(A)}_{ij}$ if $i$ is connected to $j$ with at least a path 
of length less than or equal to two.
Aggressive coarsening is usually used with polynomial interpolation \cite{Yang2010} 
which computes a second-stage interpolation matrix $P_2$ and combined with the 
first-stage $P_1$ as $P=P_1P_2$. 
The aforementioned MM-based interpolation is also available for the second stage.
Finally, Galerkin triple-matrix products are used to build
coarse-level matrices $A_c = P^TAP$ involving the prolongation $P$ 
and restriction $P^T$ operators. Falgout et al.~\cite{Falgout2020} 
provides the additional details omitted 
here on the algorithms employed in hypre for computing distributed
and parallel sparse matrix-matrix  M--M multiplications.

\subsection{Polynomial Gauss--Seidel Smoother}

To solve a linear system $A \vect{x} = \vect{b}$, 
the Gauss-Seidel (GS) iteration is based 
upon the matrix splitting $A = L + D + U$, 
where $L$ and $U$ are the strictly lower and upper 
triangular parts of the matrix~$A$, respectively.
Then, the traditional GS updates the solution based on the following recurrence,
\begin{equation}\label{eq:one-stage}
 \vect{x}^{(k+1)} := \vect{x}^{(k)} + M^{-1} \vect{r}^{(k)},  \quad k=0, 1, 2, \ldots
\end{equation}
where $\vect{r}^{(k)} = \vect{b} - A \vect{x}^{(k)}$, $A = M - N$, and $M = L + D$, 
$N = -U$ or $M = U + D$, $N=-L$ for the forward or backward sweeps, respectively. 
In the following,  $\vect{x}^{(k)}$ is the  $k$--th Gauss-Seidel iterate. 
To avoid explicitly forming the matrix inverse $M^{-1}$ in \eqref{eq:one-stage},
a sparse-triangular solve is used to apply $M^{-1}$ to the current residual
vector~$\vect{r}^{(k)}$. 

To improve the parallel scalability, hypre implements a hybrid variant of GS~\cite{Baker:2011},
where the neighboring processes first exchange the elements of the solution vector on 
the boundary of the subdomain assigned to an a parallel process (MPI rank), but then 
each MPI rank independently applies the local relaxation.
Furthermore, in hypre, each rank may apply multiple local GS sweeps for each round of the 
neighborhood communication. With this approach, each local relaxation updates only the local 
part of the vector $\vect{x}^{(k+1)}$ (during the local relaxation, the non-local solution elements 
on the boundary are not kept consistent among the neighboring ranks). 
This hybrid algorithm is shown to be effective for many problems~\cite{Baker:2011}.

A polynomial GS relaxation employs a fixed number of ``inner'' stationary iterations for approximately solving the triangular system with~$M$,
\begin{equation}\label{eq:polynomial}
 \widehat{\vect{x}}^{(k+1)} := \widehat{\vect{x}}^{(k)} + \widehat{M}^{-1} 
 (\: \vect{b} - A \widehat{\vect{x}}^{(k)} \: ), \quad k=0, 1, 2, \ldots
\end{equation}
where $\widehat{M}^{-1}$ represents the approximate triangular system solution, 
i.e., $\widehat{M}^{-1} \approx M^{-1}$.
A Jacobi inner iteration is employed to replace the direct triangular solve. 
In particular, if $\vect{g}_k^{(j)}$ denotes the approximate solution from 
the $j$-th inner iteration at the $k$-th outer GS iteration, 
then the initial solution is chosen to be the diagonally scaled residual vector,
\begin{equation}\label{eq:jr-initial-guess}
  \vect{g}_k^{(0)} = D^{-1}\vect{r}^{(k)},
\end{equation}
and the $(j+1)$--st Jacobi iteration computes the approximate 
solution by the recurrence
\begin{eqnarray}
\label{eq:jr}
  \vect{g}_k^{(j+1)} & := & \vect{g}_k^{(j)} + D^{-1} (\vect{r}^{(k)} - (L+D) \vect{g}_k^{(j)})\\
\label{eq:jacobi}
                     &  = & D^{-1}(\vect{r}^{(k)} - L \vect{g}_k^{(j)}).
\end{eqnarray}
When ``zero'' inner sweeps are performed, the polynomial GS recurrence becomes
\[
  \widehat{\vect{x}}^{(k+1)} := \widehat{\vect{x}}^{(k)} + \vect{g}_k^{(0)} = 
  \widehat{\vect{x}}^{(k)} + D^{-1}(\vect{b} - A \widehat{\vect{x}}^{(k)}),
\]
and this special case corresponds to Jacobi iteration for the global system, 
or local system on each MPI rank.
When $p$ inner iterations are performed, it follows that
\begin{align*}
\widehat{\vect{x}}^{(k+1)} &:= \widehat{\vect{x}}^{(k)} + \
vect{g}_k^{(p)} = \widehat{\vect{x}}^{(k)} + \sum_{j=0}^{p}
(-D^{-1} L)^j D^{-1} \widehat{\vect{r}}^{(k)} \\
&\approx \widehat{\vect{x}}^{(k)} + 
(I+D^{-1} L)^{-1} D^{-1} \widehat{\vect{r}}^{(k)} = 
\widehat{\vect{x}}^{(k)} + M^{-1}  \widehat{\vect{r}}^{(k)},
\end{align*}
where $M^{-1}$ is approximated by the degree-$p$ Neumann expansion and 
$\widehat{\vect{r}}^{(k)}$ is the residual following the iterated inner 
solve (e.g., $\widehat{\vect{r}}^{(k)}= \vect{b} -A\widehat{\vect{x}}^{(k)}$).
Note that $D^{-1} L$ is strictly lower triangular and nilpotent 
so that the Neumann series converges in a finite number of steps.

\subsection{ILU Smoothers with Scaled Factors and Iterated Triangular Solves}

When applying an incomplete $LU$ factorization as the smoother, careful consideration 
of algorithms which solve the triangular systems is needed.  A direct triangular solver 
is comparatively more expensive on a GPU than iterations using sparse matrix vector 
products,  and thus Jacobi iteration was proposed in \cite{Anzt2015}.  However, in cases 
when the triangular factors are highly non-normal, the Jacobi iterations may diverge
\cite{Anzt2015,Chow2015,Eiermann1993}.
Note that triangular matrices are necessarily non-normal \cite{Horn2012}, 
and three methods to mitigate the degree of non-normality are explored: 
(1) The incomplete $LDU$ factorization with row scaling, 
(2) reordering to enforce symmetric sparsity patterns, and (3) row and column scaling.  
For (2), several reordering strategies are considered, and for (3), 
the Ruiz strategy \cite{Knight2014} is considered.  
Before discussing both experimental and theoretical analyses for these approaches,
some background on (non)-normality is first provided.


\subsubsection{Jacobi Iteration and Normality}

The Jacobi iteration for solving $Ax = b$ can be written 
in the compact form below, with the regular splitting $A = M - N$,
$M = D$ and $N = D - A$. The iteration matrix $G$ is defined as
$G = I - D^{-1}A$ and
\begin{equation}
\vect{x}^{(k+1)} = G\:\vect{x}^{(k)} + D^{-1} \vect{b}
\end{equation}
or  in the non-compact form, with $M = A - D$
\begin{equation}
\vect{x}^{(k+1)} = \vect{x}^{(k)} + D^{-1}\:(\:\vect{b} - A\:\vect{x}^{(k)}\:)
\end{equation}
where $D$ is the diagonal part of $A$. For the triangular systems
resulting from the ILU factorization $A \approx LU$ (as opposed to
the splitting $A = D + L + U$), we denote the iteration matrices as $G_L$ 
and $G_U$ for the lower and upper triangular
factors, $L$ and $U$ respectively.  Let $D_L$ and $D_U$ be the diagonal parts 
of the triangular factors $L$ and $U$ and $I$ denotes the identity matrix.
Assume $L$ has a unit diagonal, then
\begin{eqnarray}
G_L & = & D_L^{-1}\:(\: D_L - L \:) = I - L \\
G_U & = & D_U^{-1}\:(\:D_U - U\:) = I - D_U^{-1}\:U \label{eq:LUiterMat}
\end{eqnarray}
A sufficient condition for the Jacobi iteration to converge is
that the spectral radius of the iteration matrix is less than
unity, or $\rho(G) < 1$. 
$G_L$ is strictly lower and $G_U$ is strictly upper triangular, 
which implies that the spectral radius of both iteration matrices is zero.
Therefore, the Jacobi iteration converges in the asymptotic sense for any triangular system. 
However, the convergence of the Jacobi method may be affected by the degree of 
non-normality of the matrix, which is now defined.

A normal matrix $A \in \mathbb{C}^{n\times n}$ satisfies 
$A^*A = AA^*$, and so a matrix that is non-normal can be defined 
in terms of the difference between $A^*A$ and $AA^*$ 
to indicate the degree of non-normality.  
Measures of nonnormality and the behavior of nonnormal matrices, 
have been extensively considered; see \cite{Ipsen1998,Henrici1962,Elsner1987,Trefethen2005}.  
In this paper,  we use Henrici's definition of the departure from normality of a matrix 
$A\in \mathbb{C}^{n\times n}$ 
\begin{equation}
    dep(A) = \sqrt{\|A\|_F^2-\|D\|_F^2},
\end{equation}
where $D\in\mathbb{C}^{n \times n}$ is the diagonal matrix 
containing the eigenvalues of $A$ \cite{Henrici1962}.  
By no means is this the only practical metric for describing the degree of nonnormality, 
however, it is particularly useful in the context of the current paper.


\subsubsection{Iterating with an LDU Factorization}

Given an incomplete $LU$ factorization, when either factoring out the $D$ matrix 
from the incomplete $LU$, 
or computing an incomplete $LDU$ factorization, the iteration matrix 
(\ref{eq:LUiterMat}) simplifies to 
$G = U_s$, where $U_s$ is a strictly upper triangular matrix (SUT). 
In particular, this leads naturally 
to a Neumann series, where for $\vect{f} = D_U^{-1}\vect{b}$

\begin{eqnarray}
    \vect{x}^{(k+1)} & = & \vect{f} - U_s\:\vect{f} + U_s^2\:\vect{f} - \dots + (-1)^k\:U_s^k\:\vect{f} 
    \nonumber \\
    & = & (\:I - U_s + U_s^2 - \dots + (-1)^k\:U_s^k\:)\:\vect{f} 
        \label{eq:Neumann} \\
    & = & (\:I-U_s\:)^{-1}\:\vect{f}. \nonumber
\end{eqnarray}
In \cite{Thomas2021}, an extensive empirical analysis shows that for certain drop tolerances, 
and  amounts of fill, $\|U_s\|_2<1$.  Thus, the convergence of the Neumann 
series is guaranteed. Even for those drop tolerances where $\|U_s\|<1$ does not hold, 
$\|U_s^p\|_2 <1$ for moderate $p$,  and further $\|U_s^p\|_2$ goes to $0$ fairly quickly.
Furthermore, for the (scaled) matrices considered in the current paper, a correlation exists 
between the maximum value (in magnitude) of $U_s^p$ and this rate of convergence. 
Figure \ref{fig:maxnorm}, compares $\max{|(U_s^p)_{ij}}|$ and $\|U_s^p\|_2$ 
for increasing values of $p$ for a PeleLM matrix dimension $N=331$ 
when computing an incomplete $LDU$ factorization.  
While the number of nonzeros in $U_s^p$ increases for $p \leq 6$, the norm rapidly decreases, 
and $nnz(U_s^p) = 0$ by $p = 39$; see Figure \ref{fig:nnzUp}.  Of course, for strictly upper
triangular matrices, $U_s^p$ is guaranteed to be the zero matrix when $p = n$, 
however, numerical nilpotence occurs much sooner (i.e., $p \ll n$).  
Building a theoretical framework 
to describe the relationship between these metrics is part of our ongoing work, 
and we refer to \cite{Thomas2021} for more empirical analyses of $\|U_s^p\|_2$.  
In particular, a relationship between larger drop tolerances and smaller fill-in 
when computing the ILU factorizations is observed when using row scaling.

\begin{figure}[hh]
\begin{center} 
\begin{subfigure}{.4\textwidth}
\includegraphics[width=\linewidth]{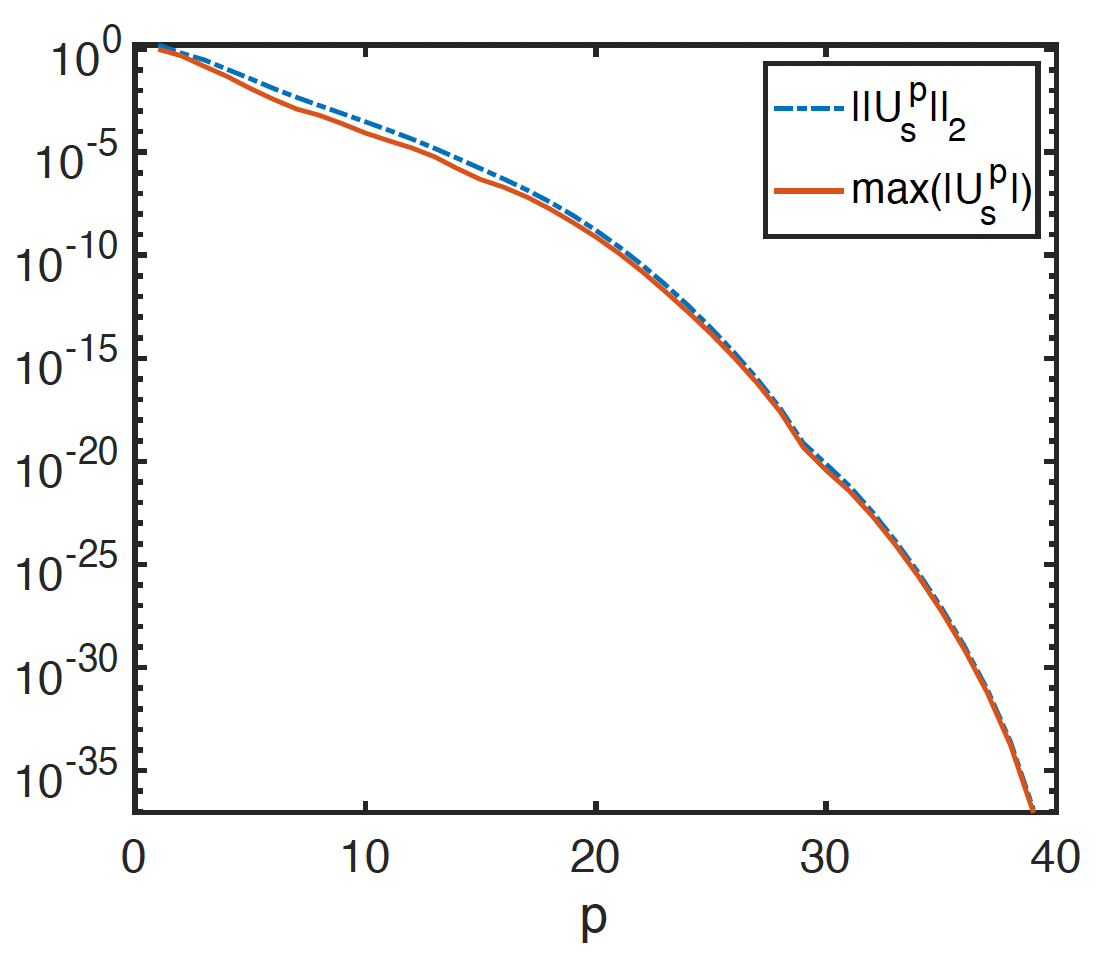}
\subcaption{$\max{|(U_s^p)_{ij}|}$ and $\|U_s^p\|_2$ for $p = 1:39$.  For $p\geq 40$, both $\max{|(U_s^p)_{ij}|}$ and $\|U_s^p\|_2$ are zero. }
\label{fig:maxnorm}
\end{subfigure}
\hspace{1cm}
\begin{subfigure}{.42\textwidth}
\includegraphics[width=\linewidth]{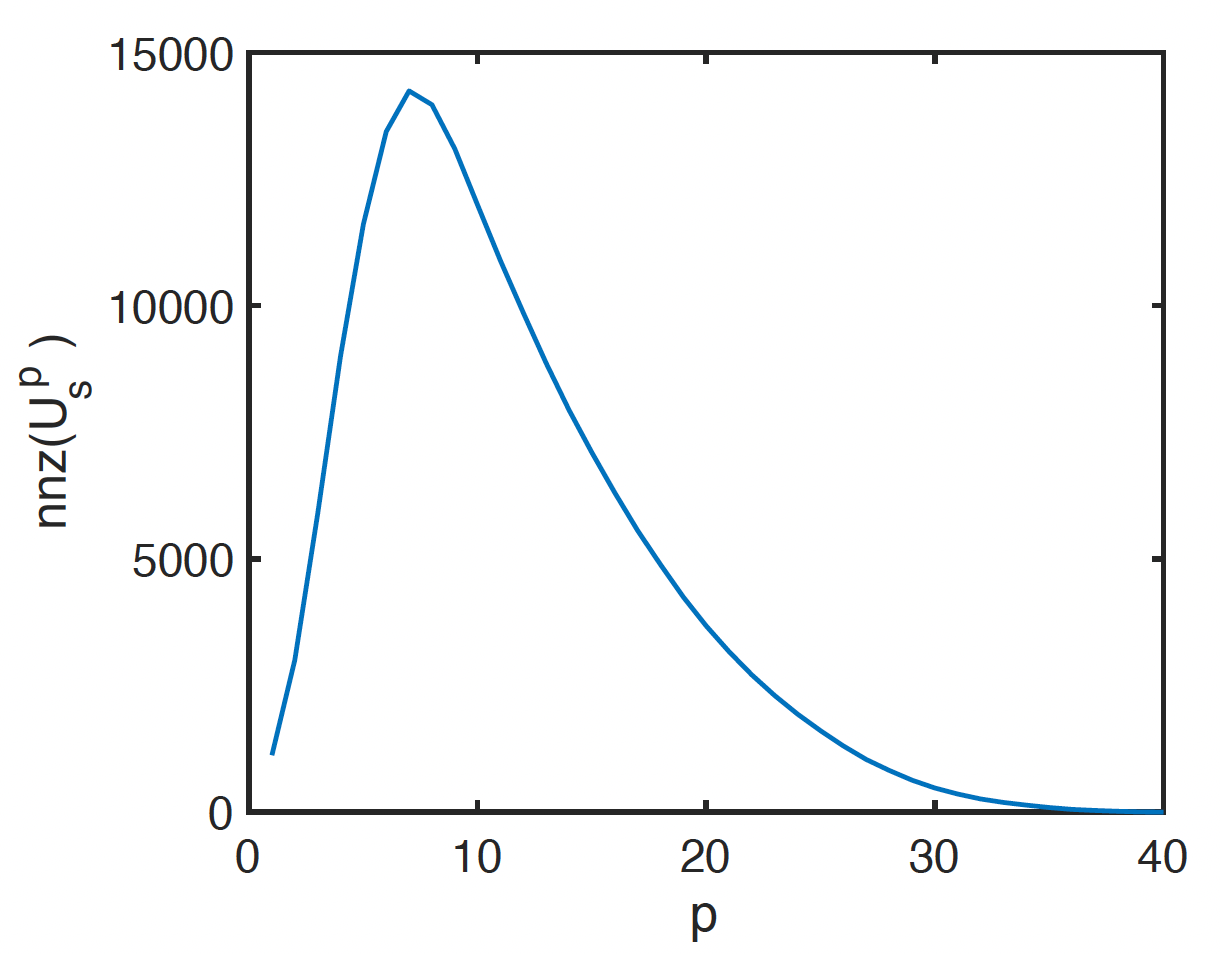}
\subcaption{Number of nonzeros in $U_s^p$ for $p = 1:39$. For $p \geq 40$, $nnz(U_s^p) = 0$.}
\label{fig:nnzUp}
\end{subfigure}
\end{center}
\caption{Comparison of max$|(U_s^p)_{ij}|$, $\|U_s^p\|_2$, and $nnz(U_s^p)$ for a linear system from PeleLM of dimension $N = 331$.}
\label{fig:exp331}
\end{figure}


\subsubsection{Matrix Re-Orderings and Non-Normality}

Richardson iteration can be interpreted in a nonlinear sense
as a continuous in space and discrete time
scheme for a heat-equation that is integrated to steady state.
The Jacobi iteration can also
be considered as a time integrator and the tools of nonlinear stability analysis 
can be applied. In particular, the recent work of Asllani et al. \cite{Asllani2021}
establishes the connection between non-normality, pseudo-spectra and the asymptotic
convergence behaviour of such nonlinear schemes.

For the purposes of iteratively solving triangular systems, one might reasonably 
presuppose that matrix re-orderings have an effect on the departure from 
normality of the $L$ and $U$ factors. In particular, re-ordering can
minimize the amount of fill-in and thereby reduce the $dep(L)$ and $dep(U)$.
Furthermore, symmetric sparsity patterns of the factors would lead to 
a similar number of Jacobi iterations for the triangular solves.
Indeed, this is observed empirically and in the case of non-symmetric
matrices, the structural symmetry resulting from the approximate minimum degree (AMD)
re-ordering in combination with a row-scaled $LDU$ factorization leads to
a small departure from normality and a rapidly converging Jacobi iteration.

Consider the Nalu-Wind pressure continuity equation matrix $A$ of dimension
$N = 3.2$ Million. Several different matrix re-orderings are applied to $A$,
followed by an incomplete $LDU$ factorization. Then the Jacobi iteration is
applied to the approximate linear system $LDUx = b$ where the right-hand side
$b$ is obtained from the extracted pressure system. 
The relative residual for each of the triangular solves
is reported in Table \ref{tab:reorder}.  The matrix re-orderings consisted of the 
symmetric reverse Cuthill-McKee (`symrcm') \cite{GeorgeLiu1981}, 
approximate minimum degree or AMD (`amd') \cite{Amestoy1996}, 
symmetric AMD (`symamd'), and nested dissection
(`dissect') \cite{George1973} in Matlab combined with an ILU$(0)$ factorization
and five (5) Jacobi iterations for the systems $Ly = b$ and $Ux = y$.

\begin{table}[htb]
\centering
\begin{tabular}{|c|c|c|c|c|} \hline
& symrcm & amd & symamd & dissect \\ \hline
$\|b - L\:y_k\|_2 / \|b\|_2$ &  0.33 & 4.2e$-5$ & 1.8e$-5$ &  8.4e$-4$ \\ \hline
$\|y - U\:v_k\|_2 / \|y\|_2$ &  15.4 & 8.1e$-5$ & 4.1e$-5$ &  8.70e$-3$\\ \hline
\end{tabular}
\caption{\label{tab:reorder} Relative residuals for $LDUx = b$
from ILU$(0)$ factorization after five (5) Jacobi iterations 
to solve $L$ and $U$ systems for a linear system from Nalu-Wind model with dimension $N=3.2$ million.}
\end{table}

Clearly, the AMD re-orderings lead to the 
lowest relative residuals and thus are candidates to be 
evaluated for a GMRES+AMG solver for the slightly non-symmetric Nalu-Wind
pressure systems. Table \ref{tab:dep-reorder} displays the departure
from normality $dep(U)$ and $dep(L)$ after re-ordering the coefficient matrix followed by
row scaling using the ILU$(0)$ factorization and row scaled $LDU$ . 

\begin{table}[htb]
\centering
\begin{tabular}{|c|c|c|c|c|c|} \hline
& unscaled & symrcm & amd & symamd & dissect \\ \hline
$dep(L)$ & 2.10e$+4$ & 1.40e$+3$ & 1.06e$+3$  &  927.3 & 1.34e$+3$
\\ \hline
$dep(U)$ & 2.01e$+4$ & 1.05e$+3$ &    787.3  &  782.4 & 796.5 
\\ \hline
\end{tabular}
\caption{\label{tab:dep-reorder} 
Henrici departure from normality $dep(U)$ for ILU$(0)$ factorization 
with row scaling $LDU$ for matrices from Nalu-Wind model with dimension $N=3.2$ million.}
\end{table}

The departure from normality without re-ordering or scaling is
$dep(U) = 2.1$e$+4$. There is a clear correlation between the
$dep(U)$ for each of the orderings in Table \ref{tab:dep-reorder}
and the size of the relative residuals in Table \ref{tab:reorder}.
Thus, a lower value for Henrici's departure from normality is a 
good predictor of the rapid convergence of the iterative triangular solves.

\subsubsection{Iterating with an LU Factorization using Ruiz Scaling}
Next, consider the application of the Ruiz strategy 
\cite{Knight2014} directly to the problematic $L$ and $U$ factors themselves. In cases where only
an $LU$ factorization is available, this strategy is appealing. The objective 
is to reduce the condition number of an ill-conditioned matrix, and it is also shown 
here to be effective in reducing $dep(U)$.  A rigorous experimental analysis of this method is 
given in \cite{Thomas2021} and in the present paper these results are augmented with theoretical
bounds.

The Ruiz algorithm is an iterative 
procedure that uses row and column scaling 
such that the resulting diagonal entries of the matrix 
are one, and all other elements are less 
than (or equal to) one \cite{Knight2014}. 
Algorithm \ref{alg:ILUTJac} describes the ILUT smoother employing the scaling for non-normal
triangular factors, and the practical implementation of this approach is described in 
more detail in \cite{Thomas2021}.
\footnote{Note that row scaling can be used in place of Ruiz scaling by instead 
constructing $D = diag(U)$, and letting $\widetilde{U} = D^{-1}U$.}  It is provided here for ease of reference.
In Section \ref{sec:results} and for the applications considered in this paper, 
only the $U$ factor is scaled. However, this algorithm may also
be extended to scale the $L$ factor in cases when both $L$ and $U$ are highly non-normal. 


\begin{algorithm}
\centering
\begin{algorithmic}[0]
    \State{Given $A\in\mathbb{C}^{n\times n}$, $\vect{b}\in \mathbb{C}^n$}
    \State{Define $droptol$ and $lfill$}
  \State{Compute $A \approx LU$ with $droptol$ and $lfill$ imposed}
  \State{Define $m_L$ and $m_u$, total number of Jacobi iterations for solving $L$ and $U$}
  \State{Define $\vect{y}^{(0)} = 0$, $\vect{v}^{(0)} =\vect{y}^{(0)}$}
  \State{//Jacobi iteration to approximately solve $L\vect{y} = \vect{b}$}
  \For{$k = 1:m_{L}$}
    \State{$\vect{y}^{(k)} = \vect{y}^{(k-1)} + (\vect{b} - L\vect{y}^{(k-1)})$ }
  \EndFor 
  \State{Apply the Ruiz strategy for $U$ and $\vect{y}$ to obtain scaled $\widetilde{U} = D_rUD_c$ and $\widetilde{\vect{y}}=D_r^{-1}\vect{y}^{(m_L)}$}
  \State{Let $D_u = diag(\widetilde{U})$}
  \State{Define $D = D_u^{-1}$}
  \State{//Jacobi iteration to approximately solve $\widetilde{U}\vect{v} =\widetilde{\vect{y}}$}
  \For{$k = 1:m_U$}
    \State{$\vect{v}^{(k)} =\vect{v}^{(k-1)} + D(\widetilde{\vect{y}} - \widetilde{U}\vect{v}^{(k-1)})$ }
  \EndFor 
  \State{//Unscale the approximate solution}
  \State{$\vect{v} =D_c^{-1}\vect{v}^{(m_U)}$}
\end{algorithmic}
\caption{\label{alg:ILUTJac} ILUT+Jacobi smoother for C-AMG with Ruiz scaling for non-normal upper triangular factors.
}
\end{algorithm}

The row and column scaling algorithm results in a transformed linear system $Ax=b$
that now takes the form
\[
L\:D_r\:U\:D_c\:\vect{x} = \vect{b}
\]
and the Jacobi iterations are applied to the upper triangular matrix $D_r\:U\:D_c$.  
Here, $D_r$ and $D_c$ represent row and column scaling, respectively, of the $U$ factor.
However, this matrix has a unit diagonal and the iterations can be expressed in
terms of a Neumann series. When an $LDU$ or $LDL^T$ factorization in available,
the series can be directly obtained from $U$ or $L^T$ as these are
strictly upper triangular matrices. The $D$ matrix can also be extracted from
either of these matrices by row scaling when possible.

When applying scaling to $U$, a matrix 
$U = I + U_s$ results, with $I$ the identity matrix of dimension $n$, and $U_s$ 
strictly upper triangular. Thus, the inverse of such a matrix can be 
expressed as a Neumann series
\begin{equation}\label{eq:Neu}
(\: I + U_s\:)^{-1} = I - U_s + U_s^2 - \cdots
\end{equation}
Because $U_s$ is upper triangular and nilpotent, the above sum is finite, 
and the series is also guaranteed to converge when $\|U_s\|_2 < 1$. 

A loose upper bound for $\|U_s\|_2$ can be derived and attributed in part 
to the fill level permitted in the ILU factorization.
\begin{theorem}\label{teo:Usf}
    Let $p$ represent the number of nonzero elements permitted in the $L$ and $U$ factors of an incomplete LU factorization.  
    Then for scaled $U\in\mathbb{C}^{n\times n}$  such that $U = I + U_s$, with $I$ the identity matrix, 
    and $U_s$ a strictly upper triangular matrix,
    \begin{equation}
        \|U_s\|_2 \leq \|U_s\|_F \leq \sqrt{n(p-1)}.
    \end{equation}
\end{theorem}
\begin{proof}
The proof follows from taking the Frobenius norm of a strictly 
upper triangular matrix 
with a maximum of $p-1$ nonzero elements per row, 
and the observation that 
$u_{ij} \leq 1$ for all $i,j$.
\end{proof}
The bound in Theorem \ref{teo:Usf} is obviously pessimistic, 
even when allowing only a small fill level in the ILU factorization. 
However, it is observed in practice that after scaling, $\|U_s\|_2 < 1$, 
and this is often true for the ILUT smoothers with $fill \leq 20$ 
for the applications considered in this paper.  
An increase in the size of $\|U_s\|_2$ is observed as the level of fill increases; 
see Table \ref{tab:2normUs} for a comparison of $\|U_s\|_2$ with $\|U_s\|_F$ 
and the theoretical bound in Theorem \ref{teo:Usf}, when varying 
the fill level.

\begin{table}[htb]
\centering
\begin{tabular}{|c|c|c|c|c|c|} \hline
& $fill = 5$ & $fill = 10$ & $fill = 20$ & $fill = 50$ & $fill = 200$ \\ \hline
$\|U_s\|_2$ & 0.679 & 0.753 & 0.87 & 1.04 & 1.22 \\ \hline
$\|U_s\|_F$ & 120.69 & 121.86 & 122.22 & 122.39 & 122.44\\ \hline
Theorem \ref{teo:Usf} Bound & 283.21 & 357.32 & 519.17 & 833.74 & 1.68e+3\\ \hline
\end{tabular}
\caption{\label{tab:2normUs} Comparison of $\|U_s\|_2$, $\|U_s\|_F$, and the 
theoretical bound provided in Theorem \ref{teo:Usf} for different amounts of fill 
permitted in the ILU factorization, and after applying scaling for a linear system from PeleLM of dimension $N=14186$.}
\end{table}

The results in Table \ref{tab:2normUs} indicate not only that the bound given in 
Theorem \ref{teo:Usf} is less pessimistic for $\|U_s\|_F$, but also (and more importantly) 
that it may be reasonably conjectured that for small levels of fill, $\|U_s\|_2 < 1 < \|U_s\|_F$.
In this case, both a lower and upper bound on $\|U_s\|_2$ may be derived as follows.

\begin{theorem}
    Let $U_s$ be defined as in Theorem \ref{teo:Usf} and 
    $r = rank(U_s)$.  
    If $\|U_s\|_2 < 1 < \|U_s\|_F$, then $\|U_s\|_2$ is bounded as
    \begin{equation}
        \frac{1}{\sqrt{r}} \leq \|U_s\|_2 < 1.
    \end{equation}
\end{theorem}
\begin{proof}
    The proof follows from the fact that $\|U_s\|_2 \leq \|U_s\|_F \leq \sqrt{r}\|U_s\|_2$.
\end{proof}

    

The rank of $U_s$ may also provide insight into the relationship between $\|U_s\|_2$ 
and the convergence behavior of the iterative method.  In particular, 
it was found that for the PeleLM matrix dimension $N = 14186$, 
and $droptol = 1$e$-1$ and $fill = 5$ for the 
ILUT factorization, ${\rm rank}(U_s) \approx 3500$. Whether or not the low-rank nature 
of $U_s$ can be exploited is a question for future investigation. For now, 
consider the largest singular values and the convergence rate of the Krylov method.  
$\|U_s\|_2$ is also equal to the largest singular value of $U_s$, denoted as $\sigma_1(U_s)$. 
The magnitude of $\sigma_1(U_s)$ is clearly related to
convergence rate of the iterative method.  In particular, 
when $\|U_s\|_2 = \sigma_1(U_s) > 1$ (e.g., when $fill > 20$ for the system considered here), 
faster convergence is achieved compared with when $\|U_s\|_2 = \sigma_1(U_s) < 1$ 
(e.g., when $fill \leq 20$); see Table \ref{tab:SV} for a comparison of the 
iterations required to reduce the relative 
residual below 1e$-5$ with $\sigma_1(U_s)$  when varying the fill. 
When the fill increases, $\sigma_1(U_s)$ grows larger, 
and the number of iterations decreases.  This is especially interesting 
because the Neumann series in (\ref{eq:Neu}) converges for $\|U_s\|_2 <1$, and yet 
the fastest GMRES convergence is observed for $\|U_s\|_2 >1$ for this particular application.  
Ongoing work seeks to build a theoretical framework to describe this relationship.  
Note here that the ILUT smoother is applied on all levels of the $V$-cycle, 
whereas in later results a 
hybrid approach is employed where the ILUT smoother is applied only at the finest level.  

\begin{table}[htb]
\centering
\begin{tabular}{|c|c|c|c|c|c|} \hline
& $fill = 5$ & $fill = 10$ & $fill = 20$ & $fill = 50$ & $fill = 200$ \\ \hline
$\sigma_1(U_s)$ & 0.679 & 0.753 & 0.87 & 1.04 & 1.22 \\ \hline
GMRES Iters & 7 & 6 & 3 & 2 & 2 \\ \hline
\end{tabular}
\caption{\label{tab:SV}Comparing the largest singular value of $U_s$ with 
 GMRES iteration counts for different fill levels permitted in the ILU factorization 
 for a linear system from PeleLM of dimension $N=14186$.}
\end{table}



Next, an upper bound on the $dep(U)$ is derived after scaling is applied, 
noting that an analogous statement can be made for $dep(L)$. 

\begin{theorem}\label{teo:dep}
Let $U \in \mathbb{C}^{n\times n}$ be a scaled upper triangular matrix 
such that $U = I + U_s$, with $I$ the identity 
matrix of dimension $n$ and $U_s$ a strictly upper triangular matrix 
with elements $u_{i,j} \leq 1$.  
For $\nu = \|U_s\|_F$ 
\begin{equation}\label{eq:depU}
    dep(U) \leq \sqrt{(2\sqrt{n} + \nu)\nu}.
\end{equation}
\end{theorem}

\begin{proof}
    For this proof, the Frobenius norm subscript is dropped for simplicity. 
    With $U = I + U_s$, it follows that
    \begin{eqnarray}
        dep(U)^2 &=& \|U\|^2 - \|D\|^2 = \|I + U_s\|^2 - n \\ 
        &= &\|I + U_s\|\|I + U_s\| - n \leq (\sqrt{n}+\|U_s\|)^2 - n\\ \label{eq:pfDepu}&=& (2\sqrt{n}+\|U_s\|)\|U_s\|.      
    \end{eqnarray}
    Denoting $\nu = \|U_s\|_F$ and taking the square root of (\ref{eq:pfDepu}), 
    the desired result is obtained.
\end{proof}
Theorem \ref{teo:dep} is completely independent of 
$\kappa_2(U)$ and gives a practical 
bound on the departure from normality of a scaled $U$. 
Consider for example $n = 10^9$ and $\|U_s\| = 0.8$; 
Theorem \ref{teo:dep} guarantees 
that after scaling, $dep(U)$ can be no larger than $225$. In general, even 
if $\|U_s\|_F$ is larger than one, as long as it is modest 
(e.g.  $O(10)$ or even $O(10^2)$),  a reasonable bound on $dep(U)$ is obtained 
after scaling the upper triangular factor. 

\begin{table}[htb]
\centering
\begin{tabular}{|l|c|c|c|c|c|} \hline
& $fill = 5$ & $fill = 10$ & $fill = 20$ & $fill = 50$ & $fill = 200$ \\ \hline
$dep(U)$ before & $7.98$e+7 & $1.06$e+8 & 1.12e+8 & 1.16e+8 & 1.17e+8 \\ \hline 
$dep(U)$ after  & 19.47 & 25.79 & 27.42 & 28.16 & 28.40 \\ \hline 
Theorem \ref{teo:dep} Bound & 208.12 & 209.48 & 209.89 & 210.08 & 210.14 \\ \hline 
\end{tabular}
\caption{\label{tab:depU}$dep(U)$ before and after Ruiz scaling, 
and the theoretical 
bound provided in Theorem \ref{teo:dep}, for different fill levels permitted 
in the ILU factorization. $N=14186$.}
\end{table}

Theorem \ref{teo:dep} may also provide a relatively inexpensive 
test for early termination
of the Ruiz algorithm.  In practice, the algorithm generally takes 
a maximum of five
iterations to converge, but the departure from normality does
not necessarily decrease for
each iteration. In particular, $dep(U)$ reaches a minimum after only the 
first or second iteration, with a slight increase for the remaining iterations.  
Thus, if at step $k$ of the Ruiz algorithm, $dep(U_s) < tol$ (where $tol$ can be 
some modest scalar multiple of the bound given in Theorem \ref{teo:dep}), 
the algorithm can be terminated early. For the systems considered here, 
this can result in a potential cost savings of up to four iterations of the Ruiz algorithm.

While scaling does not guarantee that the resulting matrix is diagonally dominant, 
when a large drop tolerance and small fill-in value are specified, 
the resulting $L$ and $U$ factors 
are nearly diagonally dominant, which we define as follows.  Here, we consider only the $U$ factor, but analogous results can be derived for $L$.

\begin{definition}\label{defn:NDD}
    Let $u_{ii}$ be the diagonal elements of $U \in \mathbb{C}^{n\times n}$ 
    for $i = 1, 2, \dots, n$ and $\delta = \max_i{\delta_i}$, where $\delta_i>0$
    is such that 
    \begin{equation}
        |u_{ii}| \geq \displaystyle\sum_{j \not=i}|u_{ij}|-\delta_i
    \end{equation} 
    for all $i = 1,2,\dots,n$. $U$ is {\it nearly (row) diagonally dominant} if and only if 
    \begin{equation}
        |u_{ii}| \geq \displaystyle\sum_{j \not = i} |u_{ij}|-\delta.    
    \end{equation}
\end{definition}
The following theorem now takes into consideration the near diagonal dominance of
$U$,  in cases where $U$ takes the form $U = I + U_s$, as is the case when 
applying scaling to the $U$ factor of an incomplete $LU$ factorization.
\begin{theorem}\label{teo:dd}
    Assume $U = I + U_s \in \mathbb{C}^{n\times n}$, for $U_s$ strictly upper triangular.  If $U$ is nearly diagonally dominant, then 
    \begin{equation}
        dep(U) \leq \sqrt{n}(1+\delta),
    \end{equation}
    where $\delta$ is as defined in definition \ref{defn:NDD}.
\end{theorem}
\begin{proof}
     \begin{eqnarray}
        dep(U)^2 &=& \|U\|^2 - \|D\|^2 = \displaystyle \sum_{i = 1}^n\sum_{j = 1}^n|u_{ij}|^2 - n \\ 
        & = & \sum_{i = 1}^n (\:|u_{ii}|^2 + \sum_{\substack{j = 1,\\j \not= i}}^n\:)-n  \\
        & \leq & \sum_{i = 1}^n(1 + (1+\delta)^2)-n\\
        & = & n(1+(1+\delta)^2)-n\\
        \label{eq:pfDDDepu}& = & n(1+\delta)^2.
    \end{eqnarray}
    Taking the square root of (\ref{eq:pfDDDepu}) gives the desired result.
\end{proof}
While this bound is not tighter than the one given in Theorem \ref{teo:dd}, it establishes the
connection between the fill level and $dep(U)$. Specifically, when a higher fill is used in the ILU
factorization, the resulting triangular factors have the potential to be far from diagonally dominant
(resulting from more nonzeros allowed in each row).  Theorem \ref{teo:dep} guarantees that as $\delta$
grows smaller (i.e., as $U$ becomes closer to diagonally dominant), the upper bound on $dep(U)$ 
decreases monotonically.

\section{Numerical Results}\label{sec:results}

To study the performance of the iterated Gauss-Seidel and ILU smoothers
in a practical setting, incompressible fluid flow 
simulations were performed with Nalu-Wind and PeleLM.
These are two of the primary fluid mechanics codes for the 
application projects chosen for the DOE Exascale Computing Project (ECP)
and used for high-fidelity simulations. We describe both next.

\subsection{Nalu-Wind Model}

Nalu-Wind solves the incompressible Navier-Stokes
equations, with a pressure projection scheme \cite{Ashesh2021}.  
The governing equations are discretized in time with a BDF-2 integrator,
where an outer Picard fixed-point iteration is employed to reduce the nonlinear
system residual at each time step.
Within each time step, the Nalu-Wind simulation time is often dominated by the
time required to setup and solve the linearized governing equations.
The pressure systems are solved using MGS-GMRES with an
AMG preconditioner, where a polynomial Gauss-Seidel
smoother is now applied as described in Mullowney 
et al.~\cite{Mullowney2021}.  Hence, Gauss-Seidel is a compute time intensive
component, when employed as a smoother within an AMG $V$-cycle.

The McAlister experiment for wind-turbine blades is an
unsteady RANS simulation of a fixed-wing, with a NACA0015 cross section, 
operating in uniform inflow. 
Resolving the high-Reynolds number boundary layer over the wing surface 
requires resolutions of ${\cal O}(10^{-5})$ normal to the surface resulting in 
grid cell aspect ratios of ${\cal O}(40,000)$. These high aspect ratios present a 
significant challenge. Overset meshes were employed \cite{Ashesh2021}
to generate body-fitted meshes for the wing and the wind tunnel geometry.
The simulations were performed for the wing at 12 degree angle of attack, 1 m chord length, 
denoted $c$, 3.3 aspect ratio, i.e., $s = 3.3c$, and square wing tip. 
The inflow velocity is $u_{\infty}= 46$ m/s, the density is $\rho_{\infty} = 1.225$
${\rm kg/m}^3$, and dynamic viscosity is $\mu = 3.756\times 10^{-5}$ kg/(m s), 
leading to a Reynolds number, $Re = 1.5\times 10^6$.
Wall normal resolutions were chosen to adequately represent the boundary layers 
on both the wing and tunnel walls. The $k - \omega$ SST RANS turbulence model 
was employed for the simulations. Due to the complexity of mesh generation, 
only one mesh with approximately 3 million grid points was generated. 

Coarsening is based on the parallel maximal
independent set (PMIS) algorithm of Luby \cite{DeSterck2006, DeSterck2008,Luby1986} 
allowing for a parallel setup phase.  
The strength of connection threshold is set to $\theta = 0.25$.  
Aggressive coarsening is applied on the
first two $V$-cycle levels with multi-pass 
interpolation and a stencil width of
two elements per row.  
The remaining levels employ M-M extended$+$i interpolation,
with truncation level $0.25$ together with a maximum
stencil width of two matrix elements per row. 
The smoother is hybrid block-Jacobi with two sweeps of polynomial
Gauss-Seidel applied locally on an MPI rank and then Jacobi 
smoothing for globally shared
degrees of freedom.  The coarsening rate for the wing simulation is roughly
$4\times$ with eight levels in the $V$-cycle for hypre.  Operator complexity
$C$ is close to $1.6$ indicating more efficient $V$-cycles with aggressive
coarsening, however, an increased number of GMRES iterations 
are required compared to standard coarsening. 
The comparison among $\ell_1$--Jacobi,  Gauss-Seidel and 
the polynomial Gauss-Seidel smoothers is shown in Figure~\ref{fig:conmat}.

\begin{figure}[htb!]
\centering
\includegraphics[width=0.5\textwidth]{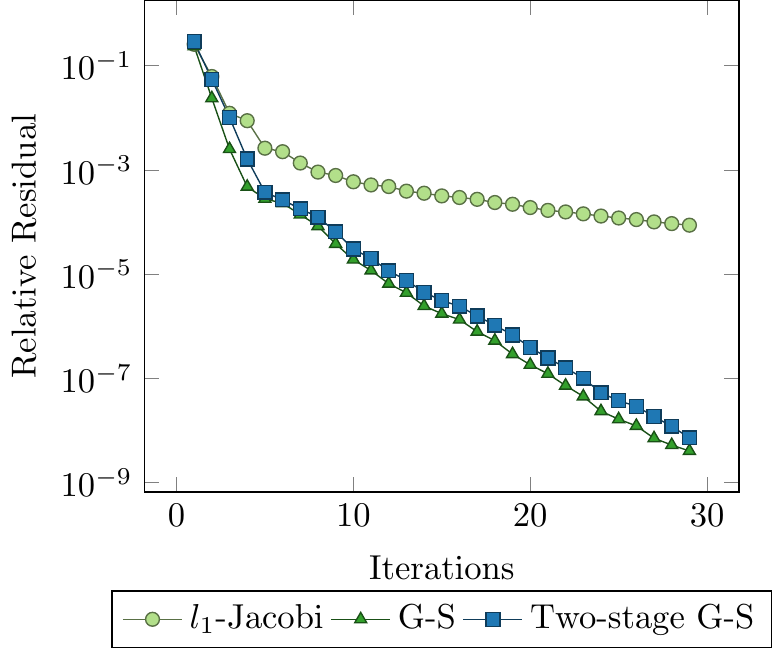}
\caption{\label{fig:conmat} GMRES$+$AMG with $\ell_1$--Jacobi, Gauss-Seidel 
and poly Gauss-Seidel smoothers for a linear system from Nalu-Wind model with dimension $N=3.1$ million.}
\end{figure}

In order to evaluate the differences between the direct and iterative 
triangular solvers employed in the smoother,
the compute times for a single GMRES$+$AMG pressure solve
are given in Table \ref{tab:smoothers}. The $\ell_1$--Jacobi
smoother from hypre is included for comparison. Both the
CPU and GPU times are reported for the NREL Eagle
supercomputer with Intel Skylake Xeon CPUs and
NVIDIA V100 GPU's. In all cases, one sweep of
Gauss-Seidel and two sweeps of $\ell_1$--Jacobi are employed
because the number of sparse matrix-vector multiplies
are equivalent in both cases. Either one CPU or GPU 
was employed in these tests. The time reported corresponds
to when the relative residual has been reduced below 1$e-$5.

\begin{table}[htb]
\centering
\begin{tabular}{|l|c|c|c|} \hline
 & $\ell_1$--Jacobi & Gauss-Seidel & Poly G-S \\ \hline
CPU (sec) & 19.4 &  12 & 19.5 \\ \hline
GPU (sec) & 0.33 & 3.1 & 0.30 \\ \hline
\end{tabular}
\caption{ \label{tab:smoothers} GMRES$+$AMG compute time. 
Jacobi, Gauss-Seidel and Poly Gauss-Seidel Smoothers for a linear system from Nalu-Wind model with dimension $N = 3.1$ million.}
\end{table}

The timing results indicate the solver time with Gauss-Seidel
is lower than when the $\ell_1$ Jacobi smoother is employed on
the CPU. However, the latter is more computationally efficient 
on the GPU. Whereas the polynomial Gauss-Seidel smoother leads 
to the lowest compute times on the GPU and results in a 
ten times speed-up compared to the smoother
that employs a direct triangular solver.

\subsection{PeleLM Model}

Pressure linear systems are taken from the ``nodal projection'' 
component of the time integrator used in PeleLM \cite{PeleLM}.
PeleLM is an adaptive mesh low Mach number combustion code 
developed and supported under DOE's Exascale Computing Program. 
PeleLM features the use of a variable-density projection scheme to ensure 
that the velocity field used to advect the state satisfies an elliptic
divergence constraint. Physically, this constraint enforces that
the resulting flow evolves consistently with a spatially uniform
thermodynamic pressure across the domain. A key feature of
the model is that the fluid density may vary considerably across
the computational domain. Extremely ill-conditioned problems arise 
for incompressible and reacting flows in the low Mach flow regime, 
particularly when using cut-cell approaches to complex
geometries, where non-covered cells that are cut by the 
domain boundary can have arbitrarily small volumes and areas.
The standard Jacobi and Gauss-Seidel smoothers are less effective 
in these cases at reducing the residual error at each level of the C-AMG
$V$-cycle and this can lead to very large iteration counts for the 
Krylov solver.

The Krylov solver was run with a hybrid $V$-cycle with ILU
smoothing only on the finest level $\ell=1$, 
and polynomial Gauss-Seidel on the remaining
levels. Iterative Jacobi triangular solvers were employed.
One pre- and post-smoothing sweep was applied on all $V$-cycle levels, 
except for the coarse level direct solve. The polynomial 
smoother leads to a slower convergence rate compared 
to the hybrid and ILU smoothers, which exhibit similar convergence. 
Because an $A\approx LDU$ factorization is employed in hypre,
Ruiz scaling was not required in our numerical experiments. 
Furthermore, only three Jacobi iterations for $L$ and $U$ were
needed to achieve sufficient accuracy to maintain the
convergence rate.

Results for the $N=1.4$ million linear system, 
representing a cylindrical geometry,
are plotted in Figure \ref{fig:14K-hypre}. These results were obtained on the
Intel Skylake CPUs and NVIDIA V100 GPUs. 
``MM-ext+i'' interpolation is employed, with a strength of
connection threshold $\theta = 0.25$.  Because
the problem is very ill-conditioned, GMRES achieves the 
best convergence rates and the lowest NRBE. 
Iterative triangular solvers were employed in these tests.
The convergence histories are plotted for hybrid ILU, 
and polynomial Gauss-Seidel smoothers in Figure \ref{fig:14K-hypre}. 
The fastest convergence rate was
achieved when the level of fill parameter was set to $lfil=10$
and with a drop tolerance of $tol = 1\times 10^{-2}$.
The reduction in the relative 
residual is sufficient for the pressure equation after 
three GMRES iterations. The ILUT parameters
were $droptol = 1$e$-2$ and $lfil = 5$. 

\begin{figure}
\centering
\includegraphics[width=0.55\textwidth]{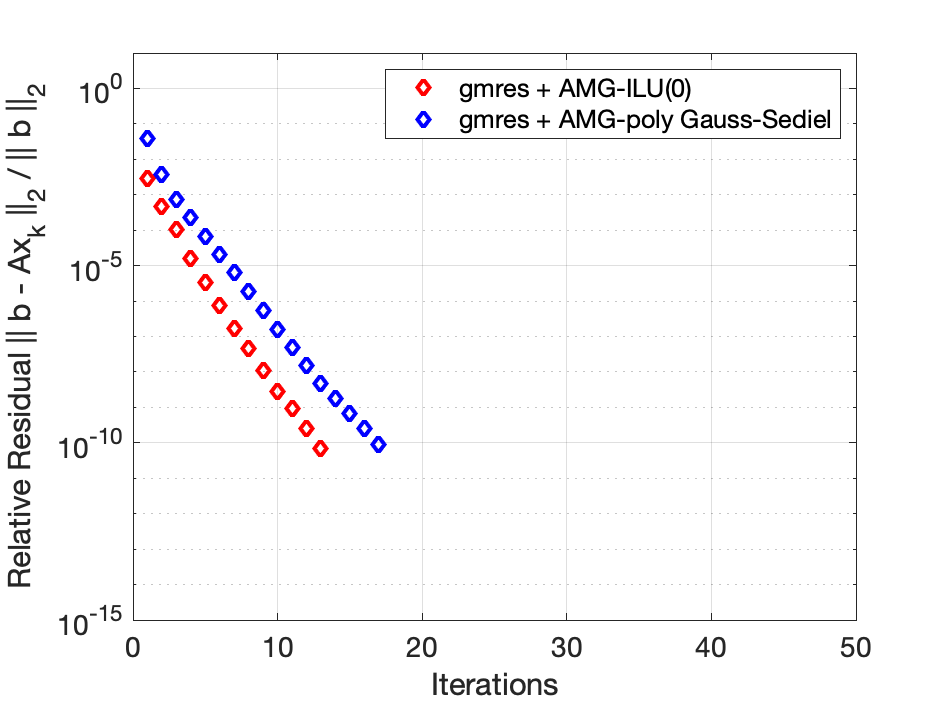}
\caption{\label{fig:14K-hypre} 
PeleLM convergence history of GMRES$+$AMG. Poly Gauss-Seidel and 
ILU$(0)$ smoothers with iterative triangular solves using three (3) Jacobi
iterations.  Matrix dimension $N=1.4$ million.}
\end{figure}

Finally, the compute times of the GMRES$+$AMG solver,
using either direct or iterative 
triangular solvers in the ILU smoother, are compared.
The compute times for a single pressure solve
are given in Table \ref{tab:ilu-smoothers}. The ILU$(0)$
and Gauss-Seidel smoothers are included for comparison. Both the
CPU and GPU times are reported. In all cases, one 
Gauss-Seidel and one ILU sweep are employed. 
The solver time reported corresponds
to when the relative residual has been reduced below 1$e-$5.

\begin{table}[htb]
\centering
\begin{tabular}{|l|c|c|c|c|c|} \hline
 &  Gauss-Seidel & Poly G-S & ILUT direct & ILUT Jacobi & ILU$(0)$ Jacobi \\ \hline
CPU (sec) &  9.2 & 9.6      & 4.3  & 6.9  & 6.8 \\ \hline
GPU (sec) &  4.5 & 0.13     & 0.17 & 0.11 & 0.086 \\ \hline
\end{tabular}
\caption{ \label{tab:ilu-smoothers} GMRES$+$AMG compute time.
Gauss-Seidel, poly Gauss-Seidel, and ILU Smoothers for a linear system from PeleLM with dimension $N=1.4$ million.}
\end{table}

Both the CPU and GPU timing results for the PeleLM nodal projection
pressure solver are compared in Table \ref{tab:ilu-smoothers}.
First, let us consider the CPU compute times for a single solve.
The results indicate that the solver time using the
hybrid $V$-cycle with an ILU smoother on the first level, with a direct 
solver for the $L$ and $U$ factors, is less than when a Gauss-Seidel
smoother is applied on all levels. One sweep of the G-S smoother is
employed in both configurations. The longer time is primarily
due to the higher number of iterations required in this case.
The ILU smoother with iterative triangular solvers is the more
efficient approach on the GPU. Despite only three (3) sparse matrix-vector 
(SpMV) products for the Jacobi iterations to solve the $L$ and $U$ 
triangular systems, the computational speed of the GPU for the
SpMV kernel is more than sufficient to overcome the direct 
sparse triangular solve.

To further explore the parallel strong-scaling behaviour of the iterative
and direct solvers within the ILU smoothers, the GMRES+AMG solver was 
employed to solve a PeleLM linear system of dimensions $N=11$ million.
The $LDU$ form of the incomplete factorization with row scaling was
again employed and ten (10) Jacobi iterations provide sufficient smoothing
for this much larger problem. The convergence histories are displayed in
Figure \ref{fig:11million-cvg}.
The linear system solver was tested on the ORNL Summit Supercomputer
using six NVIDIA Volta V100 GPUs per node, scaling out to 192 nodes,
or 1152 GPUs. Most notably, the solver with iterative Jacobi triangular
solves achieves a ten times faster compute time compared to the direct
triangular solver as displayed in Figure \ref{fig:11million}.

\begin{figure}[h!]
\centering
\includegraphics[width=0.55\textwidth,height=0.275\textheight]{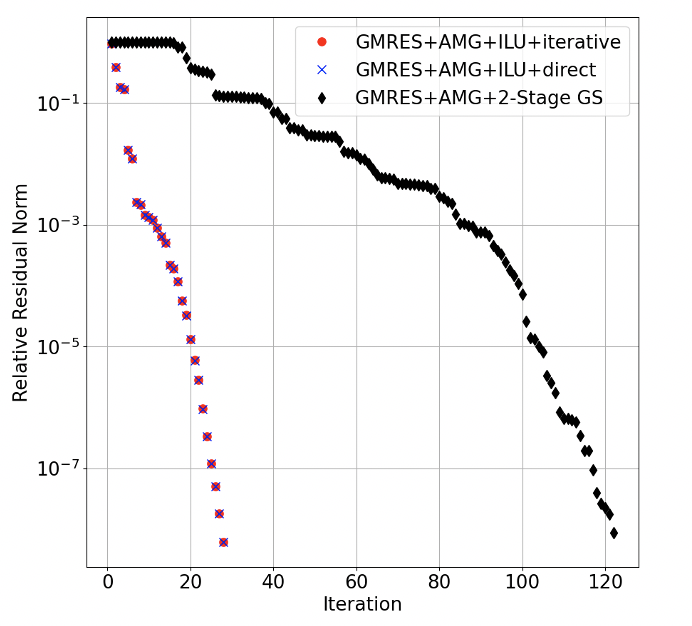}
\caption{\label{fig:11million-cvg} Hypre-BoomerAMG results.
Convergence histories of GMRES+AMG with polynomial two-stage 
Gauss-Seidel and ILU$(0)$ smoothers
using direct and iterative triangular solves for a linear system from PeleLM with dimension $N=11$ million.}
\end{figure}

\begin{figure}[hbt!]
\centering
\includegraphics[width=0.55\textwidth]{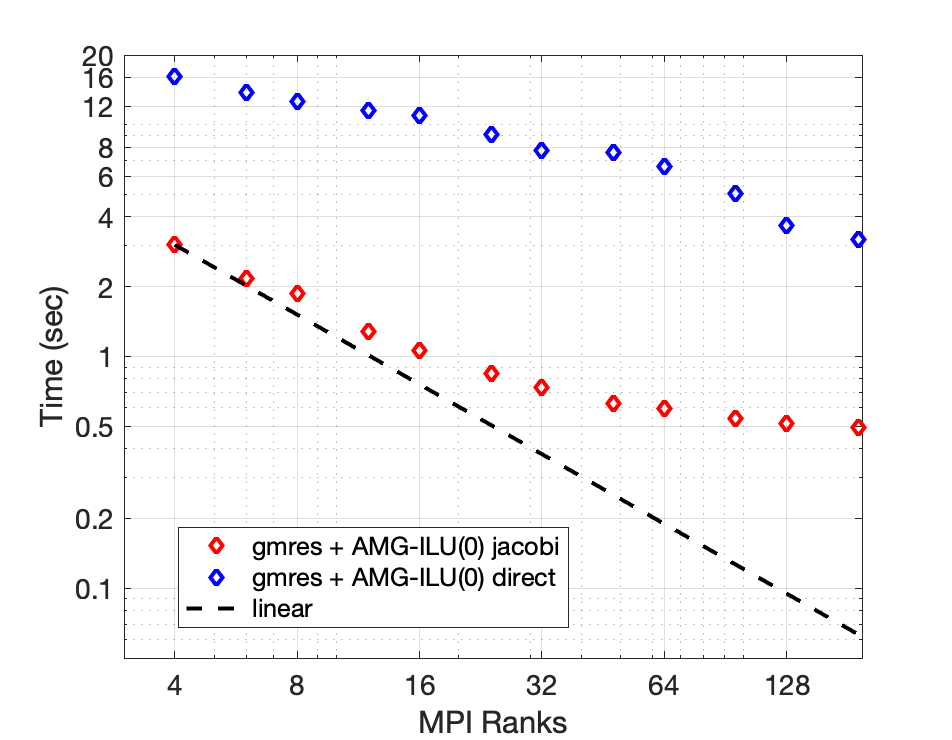}
\caption{\label{fig:11million} Hypre-BoomerAMG strong-scaling results.
GMRES+AMG with ILU$(0)$ smoothers using direct and iterative triangular solves for a linear system from PeleLM with dimension $N=11$ million.}
\end{figure}

\section{Conclusions}

We have shown how Neumann series play
an important role in a low-synch formulation of the MGS-GMRES
algorithm as well as both polynomial Gauss-Seidel and ILU
smoothers for an AMG preconditioner. 
These algorithms are well-suited to GPU
accelerators because they rely on matrix-vector products. 
Indeed, the MGS projection step and the
AMG smoother can both be accelerated by a factor of
25--50$\times$ on current GPU architectures. 
To a large extent, the loss of orthogonality of
the Krylov vectors determines when the solver will
converge to a backward stable solution. The
loss of orthogonality results elucidated by Paige and Strako\v{s}
\cite{Paige2002} imply that the strictly lower triangular 
matrix $L_k$, appearing in the correction matrix $T_k$,
remains small in the Frobenius norm and thus $T_k=I-L_k$ is
sufficient for convergence when 
$\|L_k\|_F^p = {\cal O}(\eps^p)\kappa_F^p(B)$, and $p > 1$.
The Neumann series is a finite sum because $L_k$ is strictly 
lower triangular and nilpotent. 

The iteration matrices for polynomial Gauss-Seidel and
ILU smoothers may be expressed as Neumann series.
When higher-order terms are included, rapid convergence 
is achieved for GMRES-AMG applied to the incompressible
Navier-Stokes equations pressure solver. 
A block-Jacobi type smoother is applied within each parallel process
(MPI rank) by hybrid smoothers for the local block diagonal
matrices. The approximate minimum degree (amd) matrix
re-ordering combined with an $LDU$ factorization leads
to rapid convergence of the Jacobi iterations for
the triangular factors $L$ and $U$ for slightly non-symmetric
linear systems. The Krylov solver was found
to converge faster and take less compute time when AMG
with ILU$(0)$ smoothers and iterated triangular solves
were then employed. We conjecture that re-orderings
which preserve structural symmetry of the matrix also lead 
naturally to a small departure from normality $dep(U)$ for the
triangular factors arising from an $LDU$ factorization.

Pressure continuity problems from the PeleLM adaptive mesh
combustion model was also examined. GMRES-AMG solver convergence
is essentially identical for the standard and
polynomial Gauss-Seidel smoothers. However, the
run-time on a many-core GPU accelerator drops in comparison with 
the conventional Gauss-Seidel smoother and remains
lower even when a Jacobi smoother  
is applied. The PeleLM nodal pressure projection linear systems
were highly ill-conditioned due to the variable density across
the computational domain combined with the mesh geometry. 
A sufficient reduction in the norm-wise relative backward error was
achieved using an ILU smoother with $LDU$ 
factorization and iterative solution of the triangular systems.
Despite the need for ${\cal O}(10)$ Jacobi iterations for these sparse
triangular solves, the sparse matrix-vector product (SpMV)
kernel for the GPU is much faster than the direct triangular
solver, resulting in a ten times faster compute time. 
We plan to explore the fixed-point iteration 
algorithms of Chow and Patel \cite{Chow2015} and also 
Anzt et al. \cite{Anzt2019} to 
compute the ILU$(0)$ and ILUT factorizations in order to reduce 
the C-AMG set-up costs.

\section*{Acknowledgment}
Funding was provided by the Exascale Computing Project (17-SC-20-SC).  The
National Renewable Energy Laboratory is operated by Alliance for Sustainable
Energy, LLC, for the U.S. Department of Energy (DOE) under Contract No.
DE-AC36-08GO28308.
A portion of this research used resources of the Oak Ridge Leadership Computing
Facility, which is a DOE Office of Science User Facility supported under
Contract DE-AC05-00OR22725 and using computational resources at NREL,
sponsored by the DOE Office of Energy Efficiency and
Renewable Energy.

\bibliographystyle{elsarticle-num}
\bibliography{biblio___orthogonalization}

\end{document}